\documentclass[12pt]{amsart}   	

\usepackage{geometry}                		

\usepackage{amsmath}
\usepackage{amsthm}
\usepackage{mathpazo}
\usepackage{stmaryrd}
\usepackage{mathbbol}
\usepackage{graphicx}				
\usepackage{tikz}

\usepackage{amssymb}
\usepackage{thmtools}
\usepackage{thm-restate}
\usepackage{hyperref}
\usepackage[capitalize]{cleveref}
\usepackage{longtable}
\usepackage{setspace}

\newtheorem{lemma}{Lemma}
\newtheorem{definition}{Definition}

\newtheorem{example}{Example}
\newtheorem*{example*}{Example}
\newtheorem{proposition}{Proposition}
\newtheorem{question}{Question}
\newtheoremstyle{named}{}{}{\itshape}{}{\bfseries}{.}{.5em}{\thmnote{#3}}
\theoremstyle{named}

\numberwithin{theorem}{section}
\numberwithin{lemma}{section}
\numberwithin{definition}{section}
\numberwithin{corollary}{section}
\numberwithin{definition}{section}

\newcommand{\N}{\mathbb{N}}
\newcommand{\R}{\mathbb{R}}
\newcommand{\Z}{\mathbb{Z}}
\newcommand{\Q}{\mathbb{Q}}

\newcommand{\x}{{\bf x}}
\newcommand{\ba}{{\bf a}}
\newcommand{\bu}{{\bf u}}
\newcommand{\Trop}{\textup{Trop}}
\newcommand{\ini}[1]{in_{#1}}
\newcommand{\K}{\Bbbk }

\title{Well-poised hypersurfaces}

\author[Cecil]{Joseph Cecil}
\address{Information Sciences Institute, Marina Del Rey, CA, USA 90292}
\email{jcecil@isi.edu}

\author[Dutta]{Neelav Dutta}
\address{University of Virginia, Charlottesville, VA, USA 22903-1738}
\email{nsd8uc@virginia.edu}

\author[Manon]{Christopher Manon}
\address{University of Kentucky, Lexington, USA 40506}
\email{Christopher.Manon@uky.edu}

\author[Riley]{Benjamin Riley}
\address{University of Michigan, Ann Arbor, MI, USA 48109-1382}
\email{benriley@umich.edu}

\author[Vichitbandha]{Angela Vichitbandha}
\address{North Carolina State University, Raleigh, NC, USA 27695}
\email{avichit@ncsu.edu}

\begin{document}

\begin{abstract}
An ideal $I$ is said to be "well-poised" if all of the initial ideals obtained from points in the tropical variety $\Trop(I)$ are prime.  This condition was first defined by Nathan Ilten and the third author.  We classify all well-poised hypersurfaces over an algebraically closed field.  We also compute the tropical varieties and associated Newton-Okounkov bodies of these hypersurfaces.
\end{abstract}

\maketitle

\tableofcontents

\section{Introduction}

Let $\K[\x] = \K[x_1, \ldots, x_n]$ be the polynomial ring in $n$ variables over an algebraically closed field $\K$ and let $I \subset \K[\x]$ be a monomial-free ideal. The \emph{tropical variety} $\Trop(I)$ associated to $I$ is the set of vectors $\boldsymbol \omega \in \R^n$ whose associated ideal of \emph{initial forms} $\ini{\boldsymbol \omega}(I)$  (see \cref{eq-initialform}) also contains no monomials, \cite{Maclagan-Sturmfels}. The zero locus $V(\ini{\boldsymbol \omega}(I))$ of such an initial ideal is a flat degeneration of the affine variety $V(I) \subseteq \mathbb{A}^n(\K)$.  When $\ini{\boldsymbol \omega}(I)$ is a prime binomial ideal, the variety $V(\ini{\boldsymbol \omega}(I))$ is an affine (possibly non-normal) toric variety, and we say that $\boldsymbol \omega \in \Trop(I)$ defines a \emph{toric degeneration} of $V(I)$. In this case, $\boldsymbol \omega$ is said to be a \emph{prime point} of $\Trop(I)$, and the open face $\sigma$ of the Gr\"obner fan of $I$ containing $\boldsymbol \omega$ in its relative interior (likewise contained in $\Trop(I)$) is a \emph{prime cone}. In the following, we give $\Trop(I)$ the fan structure inherited from the Gr\"obner fan of $I$, and we let $\ini{\sigma}(I)$ denote the initial ideal associated to a relatively open face $\sigma \in \Trop(I)$.

Due to their close connection with polyhedral geometry, prime binomial ideals and their associated toric varieties are often easier to handle than general prime ideals. For example, the Gorenstein property, the Cohen-Macaulay property, the Koszul property, normality of the corresponding variety, and bounds on the Betti numbers can be more easily checked for prime binomial ideals. Moreover, if these properties hold for a flat degeneration of a variety, they can be established for the variety itself. In this way toric degeneration can be a useful tool for studying both the geometry of the original variety $V(I)$ and its coordinate algebra $\K[\x]/I$.   In particular, it can be shown that there is a \emph{Newton-Okounkov body} (\cite{KK}, \cite{LM}, \cite{Ok}) associated to each prime cone of maximal dimension in $\Trop(I)$ (\cite{Kaveh-Manon-NOK}) from which many invariants of $V(I)$ can be extracted.  Recently, Escobar and Harada \cite{Escobar-Harada} have shown that maximal prime cones in $\Trop(I)$ which share a facet give rise to a \emph{wall-crossing} phenomenon between their associated Newton-Okounkov bodies. For this reason it is of interest to know when $\ini{\boldsymbol \omega}(I)$ is a prime ideal for every $\boldsymbol \omega \in \Trop(I)$. Following work of the third author and Ilten in \cite{Ilten-Manon} such an ideal is said to be \emph{well-poised}. In this paper, we classify all well-poised principal ideals (Theorem \ref{maintheorem}). A description of Newton-Okounkov bodies for well-poised hypersurfaces appears in Section \ref{Angela}.  

 We write $f = \sum c_i \x^{{\bf a}_i}$ to mean a polynomial in $\K[\x]$ with monomial terms $c_i\x^{{\bf a}_i}$, for $c_i \in \K$ and $\x^{{\bf a}_i} = x_1^{a_{i,1}}\cdots x_n^{a_{i,n}}$. The initial form:
 
 \begin{equation}\label{eq-initialform}
 \ini{\boldsymbol \omega}(f) = \sum_{{\bf a}_i \in M} c_i\x^{{\bf a}_i}
 \end{equation}

 \noindent
 for a real vector $\boldsymbol \omega = (\omega_1, \ldots, \omega_n) \in \R^n$ is the sum of those monomial terms $c_i\x^{{\bf a}_i}$, where ${\bf a}_i$ belongs to the set $M$ of those exponents whose inner product with $\boldsymbol \omega$ is maximal (see \cite{Sturmfels96}).  Likewise, the initial ideal $\ini{\boldsymbol \omega}(I) \subset \K[\x]$ is the ideal generated by the initial forms $\{\ini{\boldsymbol \omega}(f) \mid f \in I\}$.  If $I$ is principal, say generated by $f \in \K[\x]$, then $\ini{\boldsymbol \omega}(I) = \langle \ini{\boldsymbol \omega}(f) \rangle$; with this in mind we get the following definition. 
 
\begin{definition}\label{wellpoiseddefinition}
A polynomial $f$  is said to be {\bf well-poised} if every initial form which is not a monomial is irreducible.
\end{definition}

\noindent
We introduce the following conventions. We say that that $\gcd({\bf a}_i, {\bf a}_j)$ is the gcd of all entries of the two exponent vectors. We recall the \emph{support}  of a monomial term $\x^{{\bf a_i}}$, denoted $supp(\x^{{\bf a_i}})$, is the set $\{j : {\bf a}_{i,j} \neq 0\}$. Several of our results will involve the condition that $supp(x^{{\bf a}_i})\cap supp(x^{{\bf a}_j})=\varnothing$ for two monomials in a polynomial $f$, and so we give the following definition:
\begin{definition}\label{def-disjointly}
We say a polynomial $f = \sum_{i =1}^n c_i\x^{{\bf a}_i}$ is disjointly supported if $supp(x^{{\bf a}_i})\cap supp(x^{{\bf a}_j})=\varnothing$ for all $c_i, c_j \neq 0$.
\end{definition}

The following is our main result.

\begin{restatable}{theorem}{maintheorem}
\label{maintheorem}
A polynomial $f = \sum_{i \in \mathbb{N}}c_i {\bf x}^{{\bf a}_i}$ is well-poised if and only if $f$ is disjointly supported and $\gcd({\bf a}_i, {\bf a}_j) = 1$ for any pair $i,j\in \N$.
\end{restatable}

We also give a complete description of two combinatorial invariants of a well-poised hypersurface. We show that the \emph{Newton polytope}, denoted $N(f)$, of any well-poised hypersurface contains no interior lattice points (\cref{interiorlatticepoints}), and we give a complete description of the tropical variety $\Trop(f)$ (\cref{Angela}).

\begin{restatable}{theorem}{interiorlatticepoints}
\label{interiorlatticepoints}
Let $f$ be well-poised. Then $N(f)$ is a simplex. Further, $N(f)\cap \Z^n$ is precisely the vertex set of $N(f)$. \end{restatable}

In Section \ref{Angela} we determine the structure of the tropical variety of a polynomial with disjoint supports. We would also like to note here that for these disjointly supported hypersurfaces, (and consequently well-poised hypersurfaces), the computation of the singular locus is a straightforward excerise in computing the Jacobian. By examining these conditions, one can easily compute the codimension of the singular locus and determine normality by applying Serre's criterion.

\begin{example}[E8 Singularity]
\label{E8}
\emph{
The Du Val $E_8$ singularity is given by the solution set of $x^2 + y^3 + z^5=0$. This is both  well-poised and normal. The hypersurface $x^2 + y^3 + z^5=0$ defines a normal, rational, complexity-1 $T$-variety as studied in \cite{Ilten-Manon}. In particular, this is a \emph{semi-canonical} embedding of the hypersurface, and is therefore well-poised by the results in \cite{Ilten-Manon}.  
}
\end{example}

\begin{example}[The Grassmannian $Gr_2(4)$]
\label{Grass}
\emph{
The affine cone over the Pl\"ucker embedding of the Grassmannian variety $Gr_2(4)$ is given by the solution set of  $p_{12}p_{34} - p_{13}p_{24} + p_{14}p_{23}=0$ in $\bigwedge^2(\K^4)$. A simple application of our results show that this is well-poised. 
}
\end{example}

Polynomials of the type described in Theorem \ref{maintheorem} have appeared recently in work of Hausen, Hische, and Wr\"obel on Mori dream spaces of low Picard number (\cite{HHW}).  The Cox rings of smooth general arrangement varieties of true complexity $2$ and Picard number $2$ of all types except $14$ are presented by well-poised hypersurfaces. Type $14$ is well-poised as well, but it is not a hypersurface.  As a consequence, any projective coordinate ring of one of these varieties carries a full rank valuation with finite Khovanskii basis for each maximal cone of the tropicalization of the hypersurface, and the varieties themselves have a toric degeneration for each maximal cone of the tropical variety of the hypersurface. We illustrate this construction with an example of a singular del Pezzo surface. We remark that the spectra of these Cox rings are also arrangement varieties, and so can be proved to be well-poised by recent results of Joseph Cummings and the third author \cite{Cummings-Manon} by a different method. 

\begin{example}
\emph{
We consider the example of the singular, $\Q$-factorial Gorenstein del Pezzo surface $X$ described in \cite[Example 3.2.1.6]{Cox-book}. The Cox ring of $X$ is presented by a well-poised hypersurface:}
\[Cox(X) = \K[T_1,T_2,T_3,T_4,T_5]/\langle T_1T_2+T_3^2+T_4T_5\rangle.\]

\noindent
\emph{The lineality space of $T_1T_2+T_3^2+T_4T_5$ has basis given by the rows of the matrix:}

\[M = \begin{bmatrix}
1 & 1 & 1 & 1 & 1\\
1 & -1 & 0 & -1 & 1\\
1 & 1 & 1 & 0 & 2\\
\end{bmatrix}.\]

\emph{
\noindent
By Theorem \ref{maintheorem}, we obtain a toric degeneration of $Cox(X)$ for each of the three maximal cones of $\Trop(\langle T_1T_2+T_3^2+T_4T_5 \rangle)$.  The corresponding \emph{global Newton-Okounkov bodies} of $X$ are obtained as the $\Q_{\geq 0}$ spans of the columns of the matrices $M_1, M_2, M_3$ obtained by appending $w_1 = (1, 1, 1, 0, 0)$, $w_2 = (1, 1, 0, 1, 1)$, or $w_3 = (0, 0, 1, 1, 1)$ to the bottom of $M$, respectively. }

\emph{The grading on $Cox(X)$ defined by the 2nd and 3rd rows of $M$ coincides with the class group grading. Let $M'$ be the matrix given by these rows, so that the effective classes on $X$ coincide with the $\Z_{\geq 0}$ span of the columns of $M'$. For $(i, j)$ we let $\Gamma(i,j) \subset Cox(X)$ be the $(i, j)$-th graded component.  The polynomial ring $S = \K[T_1, T_2, T_3, T_4, T_5]$ also inherits a grading by the columns of $M'$, and for each $(i, j)$ there is a surjection $S(i, j) \to \Gamma(i, j) \to 0$. The initial forms $in_{w_i}(T_1T_2+T_3^2+T_4T_5)$ are each homogeneous with respect to the grading by $M'$, as a consequence $S(i, j)$ also surjects onto the $(i, j)$-th component of each degeneration $\K[T_1, T_2, T_3, T_4, T_5]/\langle in_{w_i}(T_1T_2 + T_3^2 + T_4T_5)\rangle$.}  

	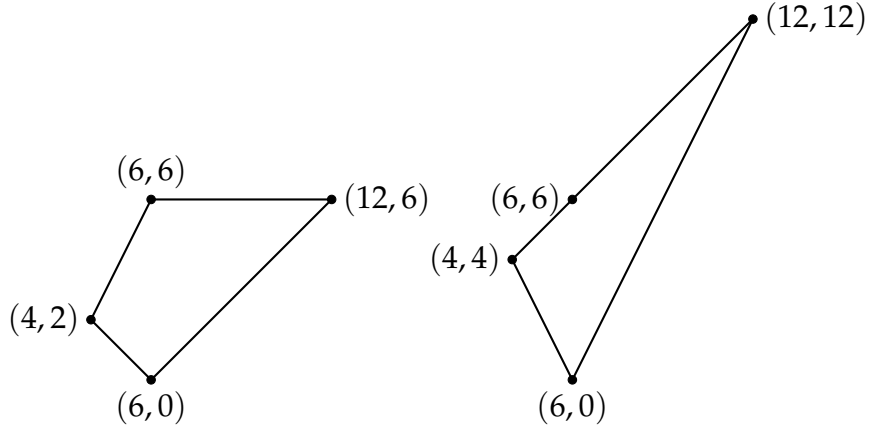
\begin{figure}
		\centering
		\begin{tikzpicture}[scale=.8]
		
		
		\draw[thick] (3,0) coordinate (a_1) -- (2,1) coordinate (a_2);
		\draw[thick] (a_2) -- (3, 3) coordinate (a_3);
		\draw[thick] (a_3) -- (6, 3) coordinate (a_4);
		\draw[thick] (a_4) -- (a_1);
		\filldraw [black] (a_1) circle (2pt) node [anchor=north] {$(6, 0)$};
		\filldraw [black] (a_2) circle (2pt) node [anchor=east] {$(4, 2)$};
		\filldraw [black] (a_3) circle (2pt) node [anchor=south] {$(6, 6)$};
		\filldraw [black] (a_4) circle (2pt) node [anchor=west] {$(12, 6)$};
		
		\draw[thick] (10,0) coordinate (b_1) -- (9,2) coordinate (b_2);
		\draw[thick] (b_2) -- (10, 3) coordinate (b_3);
		\draw[thick] (b_3) -- (13, 6) coordinate (b_4);
		\draw[thick] (b_4) -- (b_1);
		\filldraw [black] (b_1) circle (2pt) node [anchor=north] {$(6, 0)$};
		\filldraw [black] (b_2) circle (2pt) node [anchor=east] {$(4, 4)$};
		\filldraw [black] (b_3) circle (2pt) node [anchor=east] {$(6, 6)$};
		\filldraw [black] (b_4) circle (2pt) node [anchor=west] {$(12, 12)$};

		
		
		\end{tikzpicture}
		\caption{Newton-Okounkov bodies $\Delta_1(X, (0, 6))$ and $\Delta_2(X, (0, 6))$.}
		\label{NOKbodies}
	\end{figure}

\emph{The divisor class $(0, 6)$ is ample; let $R = \bigoplus_{n \geq 0} \Gamma(0, 6n)$ be its projective coordinate ring.  The space $S(0, 6)$ has basis given by the $23$ monomials $T^{\bf a}$ such that ${\bf a} = (a_1, a_2, a_3, a_4, a_5) \in \Z_{\geq 0}^5$ satisfies $a_1 -a_2 -a_4 + a_5 = 0$ and $a_1 + a_2 + a_3 + 2a_5 = 6$.  The subring $S_{0, 6} = \bigoplus_{ n \geq0} S(0, 6n) \subset S$ can be shown to be generated by the component $S(0, 6)$. As a consequence, $R$ and each of the three degenerations \[gr_i(R) \subset \K[T_1, T_2, T_3, T_4, T_5]/\langle in_{w_i}(T_1T_2 + T_3^2 + T_4T_5)\rangle,\ i \in \{1, 2, 3\}\] are generated by their $(0, 6)$ components as well. It follows that the image of the $23$ monomials in $\Gamma(0, 6) \subset R$ are a Khovanskii basis for each of the three valuations defined by the matrices $M_1, M_2$ and $M_3$ on $R \subset Cox(X)$.  The ideal of relations which vanishes on these generators is generated by $117$ elements: $4$ linear, $112$ quadratic, and $1$ sextic.  Two of the three Newton-Okounkov bodies $\Delta_i(X, (0, 6))$ $i \in \{1, 2, 3\}$ of the class $(0, 6)$ emerging from this construction are depicted in Figure \ref{NOKbodies}, and $\Delta_3(X, (0, 6)) = \Delta_1(X, (0, 6))$. Each body $\Delta_i(X, (0, 6))$ is the image of the poytope $P(0, 6) \subset \Q_{\geq 0}^5$ cut out by the relations $a_1 -a_2 -a_4 + a_5 = 0$ and $a_1 + a_2 + a_3 + 2a_5 = 6$ under the linear projection defined by the matrix with rows $(1, 1, 1, 1, 1)$ and $w_i$, respectively. Both Newton-Okounkov bodies have volume $24$.}
\end{example}

The Cox ring of a projective variety with a finitely generated and free class group is known to be factorial \cite{Arz}. With this in mind, we ask the following question.

\begin{question}
For $f$ well-poised, when is the ring $\K[\x]/\langle f \rangle$ factorial?  In other words, when does $f$ well-poised present the Cox ring of a Mori dream space?
\end{question}

\noindent
An answer to this question would provide a large class of Cox rings of Mori dream spaces, each equipped with a large combinatorial class of toric degenerations. 

\subsection{Acknowledgements}

We thank David Ma and Alston Crowley for many useful conversations.  We also thank the UK Math Lab for hosting this project in the spring and fall of 2018.  The third author was supported by both the NSF (DMS-1500966) and the Simons Foundation (587209) during this project.

\section{The Newton Polytope and Supporting Lemmas} \label{ProofIL}
Here we prove the  results necessary to establish \cref{maintheorem} and \cref{interiorlatticepoints} in \cref{section3}.  These proofs rely on the properties of the Newton polytope of a hypersurface. Terminology and notation are taken from \cite{Maclagan-Sturmfels}. 

\begin{definition} 
	The Newton polytope $N(f)$ of a polynomial $f = \sum c_i \x^{\ba_i}$ is the convex hull of the set $\{{{\bf a}_i} : c_i \neq 0\} \subset \R^n$. \cite[61]{Maclagan-Sturmfels} 
\end{definition}

Recall that the faces of $N(f)$ are in one-to-one correspondence with the initial forms $\ini{w}(f).$ In particular, each face is of the form $N(\ini{w}(f))$ for some weight vector ${w}$, and if $N(f)$ is a simplex with no interior lattice points, then every sub-sum of $f = \sum c_i\x^{{\bf a_i}}$ is an initial form of $f$.  Also, recall that if $f=pq$ then $N(f) = N(p) + N(q)$ where the right side denotes the Minkowski sum of the Newton polytopes (\cite{Sturmfels96}). The  lemmas in this section serve to restrict the combinatorial type of $N(f)$ when $f$ is well-poised.

\begin{restatable}{lemma}{WellPDis}
	\label{WellPDis}
	If $f$ is well-poised, then all monomials corresponding to the vertices of $N(f)$ have disjoint supports. 
\end{restatable}

\begin{proof}
	Let $x_j$ be an indeterminant in $f$ and let \[S:=\{{\bf a_i} : j\in supp( \x^{\bf a_i}), {\bf a_i} \text{ is a vertex}\}\] denote the corresponding vertex set in $N(f)$, corresponding to monomials that contain $x_j$. Note that this set exclusively considers vertices and not interior lattice points. We show that this set contains only one element. 
	
	First, note that any face of $N(f)$ corresponds to an initial form of $f$. Should any of the vertices ${\bf a_i} \in S$  share a (potentially degenerate) edge, this would correspond to a factorable initial binomial (or more general form, if the edge is degenerate) of $f$, as the lowest power of $x_j$ could be factored out. Therefore for $f$ to be well-poised, no vertices in $S$ may share an edge. Let ${\bf e_j}$ be the unit basis vector with respect to the $j^{th}$ coordinate and choose a vertex $\overline{\bf{{a_i}}}$ in $S$  such that the interior product $\langle \overline{\bf{a_i}}, \bf{e_j}\rangle$ is maximal. Now, consider the edges from $\overline{\bf{{a_i}}}$ to its adjacent vertices. Notice that by our above reasoning, $\overline{\bf{{a_i}}}$ cannot connect to any other member of $S$, so all edges connect to vertices which lie in the subspace of $\R^n$ homeomorphic to $\R^{n-1}$ corresponding to $x_j=0$.
	
	We can conclude now that $\overline{\bf{a_i}}$ is the only element in $S$. If we suppose there is some other vertex in $S$, $\mathbf{a'}$, we see we are forced to choose between convexity or irreducibilty, as if these two vertices are not joined by some ray, then the polytope is not convex, but if they are joined, then the edge corresponds to a reducible initial form. Therefore, there can only be one element in $S$, and our monomials will have disjoint supports.

\end{proof}

\begin{restatable}{lemma}{PolyDis}
	\label{PolyDis}
	If the monomials corresponding to vertices in $N(f)$ have disjoint supports, then $N(f)$ is a simplex. 
\end{restatable}


\begin{proof}
Since the monomials have disjoint supports, the corresponding exponent 
vectors are affinely independent - thus $N(f)$ is the convex hull of an 
affinely independent set of points, i.e a simplex. 
\end{proof}

\cref{RedBinom} and \cref{interiorlatticepoints} require the following lemma, which restricts the number of lattice points in $N(f)$, when $f$ is of a specific form. The proof of this lemma is reserved for the end of this section to improve readability. 

\begin{restatable}{lemma}{VertPoints}
\label{VertPoints}
Let $f$ be a polynomial such that the vertices of $N(f)$ have disjoint supports. If $f$ also has $\gcd({\bf a_i, a_j})=1$ for all pairs of vertices $\bf{a_i, a_j}$, then $N(f)$ contains no lattice points besides its vertices.
\end{restatable}
Now, we prove \cref{RedBinom}. 
\begin{restatable}{lemma}{RedBinom}
	\label{RedBinom}
	A binomial $f=c_{i} \x^{\bf a_i} + c_{j} \x^{\bf a_j}$ is irreducible if and only if $\gcd(\bf a_i, a_j)=1$ and $supp(x^{{\bf a}_i})\cap supp(x^{{\bf a}_j})=\varnothing$.
\end{restatable}
\begin{proof}

	$\Leftarrow$ Consider the Newton polytope $L=N(f)$. This is a line segment with end points $\mathbf{a}_i$ and $\mathbf{a}_j$. Now suppose $f=pq$. We will show that one of the factors, say $q$, must be a constant. We must have $L=N(f)=N(p)+N(q)$.  As $f$ is a binomial in $\K[\mathbf{x}]$,  \[max(dim(N(p)),dim(N(q))) \leq dim(N(f))\leq 1\] \noindent we conclude that $N(q)$ is a line segment, a point distinct from the origin, or the origin itself. We assume without loss of generality that $N(p)$ is a line segment. If $N(q)$ is a point $\mathbf{a_0}$, then the monomial $\x^{{\bf a}_0}$ must divide $f$, so $\bf{a_i}$ and $\bf{a_j}$ do not have disjoint supports unless $a_0 = 0$.  If $N(p)$ and $N(q)$ are both lines, they must be colinear, as otherwise, $N(f)$ would be two-dimensional. The Minkowski sum of two colinear lines with integer endpoints must contain an interior lattice point, corresponding to the sum of one endpoint from each line. However, $f$ satisfies the form required in \cref{VertPoints}, and thus contains no interior lattice points. Therefore, $N(q)$ can only be the point at the origin, meaning that $q$ must be a constant. \\

	$\Rightarrow$
	If $f$ is irreducible, then $\x^{\bf a_i}$ and $\x^{\bf a_j}$ must have disjoint supports, as if they did not, we could factor out a power of $x_k$, where $k\in supp(\x^{\bf a_i})\cap supp(\x^{\bf a_j})$. Now, suppose $\gcd({\bf a_i, a_j})=d>1$, and ${\bf a'_i}d={\bf a_i}$, and ${\bf a'_j}d={\bf a_j}$. Now by rearranging constants, and factoring we get the following:
	
	\[
	f=c_{i} \x^{\bf a_i} + c_{j} \x^{\bf a_j}=c_i\x^{{\bf a'_i}d} - c \x^{{\bf a'_j}d}=(\sqrt[d]{c} \x^{\bf a'_j})^{d}((\frac{\sqrt[d]{c_i}\x^{{\bf a'_i}}}{\sqrt[d]{c}\x^{{\bf a'_j}}})^d - 1)
	\]
	Now, if we relabel $(\frac{\sqrt[d]{c_i}\x^{{\bf a'_i}}}{\sqrt[d]{c}\x^{{\bf a'_j}}})$ as $z$, the above simplifies to 
	\[f=c \x^{{\bf a'_j}d}(z^d - 1)\]
	where $z^d-1$ easily factors as $\Pi_{1<k<d}(z-\delta_k)$, where $\delta_k$ is a $d^{th}$ root of unity.  Now, by resubstituition and distributing, we get:
	\[
	\begin{split}
	f&=(\sqrt[d]{c} \x^{\bf a'_j})^{d}(z^d - 1)\\&=(\sqrt[d]{c} \x^{\bf a'_j})^d \Pi_{1<k<d}(z-\delta_k)\\
	&= \Pi_{1<k<d} ((\sqrt[d]{c} \x^{\bf a'_j})z-\delta_k \sqrt[d]{c} \x^{\bf a'_j}))\\&= \Pi_{1<k<d} (\sqrt[d]{c_i}\x^{\bf a'_i}-\delta_k \sqrt[d]{c} \x^{\bf a'_j}))
	\end{split}
	\]
	which gives that $f$ is factorable. 
	
\end{proof}

Now we return to the proof of \cref{VertPoints}, where we need the following lemma, whose proof is straight forward:
\begin{lemma}
	For a set of equivalent fractions $\{ \frac{b_1}{a_1}, \dots, \frac{b_s}{a_s} \}$, the fractions are equal to $\frac{\gcd(\bf b_i)}{\gcd(\bf a_j)}$, 	where $\gcd(\bf b_i)$ is the gcd of $b_1, b_2, \cdots, b_s$ and $\gcd(\bf a_j)$ is the gcd of $a_1, a_2, \cdots, a_s$. 
\end{lemma}

	
Recall the statement: \VertPoints*

\begin{proof}[Proof of Lemma \ref{VertPoints} ]
	We will prove that the existence of a lattice point in a polytope with vertices $\mathbf{a_1},\dots,\mathbf{a_k}$ having disjoint supports implies that $\gcd(\mathbf{a}_i, \mathbf{a}_j) > 1$ for some $i$ and $j$. Now suppose an interior lattice point $\mathbf{b}$ exists in a polytope with disjoint supports.  Then it is of the form
	\[ \mathbf{b}=\sum^k_{j=1}p_j \mathbf{a}_j.
	\]
	As the point $\mathbf{b}$ lies in the convex hull of our polytope, we have the added restriction that 
	\[\sum^k_{j=1}p_j=1.\]
	As $\mathbf{b}$ is an integer lattice point, each coordinate must be an integer. By assumption, the collection of all $\mathbf{a}_j$ have disjoint supports, so we can break up $\mathbf{b}$ into a sum of $\mathbf{b_j}$ where each $\mathbf{b}_j$ has the same support as $\mathbf{a}_j$ and thus for $i\in supp(\mathbf{a_j})$ the entry $b_{j,i}=p_j a_{j,i}$. Therefore, $p_j=\frac{b_{j,i}}{a_{j,i}}$. This also gives that for all $s$ supports in a given $\mathbf{a_j}$,  \[p_j=\frac{b_{j,1}}{a_{j,1}}=\dots=\frac{b_{j,s}}{a_{j,s}}\]
	which, by our above lemma gives $p_j=\frac{\gcd(\mathbf{b}_j)}{\gcd(\mathbf{a}_j)}$. We may now rewrite the second summation above as follows:

	\[
	\sum^{k-1}_{j=1}\frac{\gcd(\mathbf{b}_j)}{\gcd(\mathbf{a}_j)}=\frac{\gcd(\mathbf{a}_k)-\gcd(\mathbf{b}_k)}{\gcd(\mathbf{a}_k)}
	\]
	
	Now by giving the left above a common denominator of $\Pi^{k-1}_{j=1}\gcd(\mathbf{a}_j)$ the above becomes:
	\[ 
	\frac{\sum^{k-1}_{i=1}\gcd(\mathbf{b}_i)\Pi_{j\neq i}\gcd(\mathbf{a}_j) }{\Pi^{k-1}_{j=1}\gcd(\mathbf{a}_j)}=\frac{\gcd(\mathbf{a}_k)-\gcd(\mathbf{b}_k)}{\gcd(\mathbf{a}_k)}
	\]
	Now once again, this is of the form where we may use the above lemma. For the sake of notation we will relabel the above numerators $B_1$ and $B_2$ so we see that the above fraction reduces to:
	\[\frac{\gcd(B_1,B_2)}{\gcd(\Pi^{k-1}_{j=1}\gcd(\mathbf{a}_j),\gcd(\mathbf{a}_k))}.\]
	Now if we examine the denominator, this must be greater than one, as the fraction is less than one and greater than zero. This would imply that that $\gcd(\mathbf{a}_k,\mathbf{a}_i)>1$ for some $i$,  thus proving the statement. 
\end{proof}

\section{Proof of Theorem \ref{maintheorem} and \ref{interiorlatticepoints}}
\label{section3}
We can now proceed with the proofs of \cref{maintheorem} and \cref{interiorlatticepoints}. These are largely corollaries of the lemmas given in the previous section. We first prove \cref{interiorlatticepoints}, which we will use in the proof of \cref{maintheorem}. Recall the statement:
\interiorlatticepoints*
\begin{proof}
	If $f$ is well-poised, then by Lemma \ref{WellPDis} the vertices of $N(f)$ have disjoint supports satisfying the first condition. Additionally by Lemma \ref{PolyDis}, $N(f)$ is a simplex. Since $N(f)$ is a simplex, for any two vertices $\bf{x^{a_i}, x^{a_j}}$ there is an edge connecting them corresponding to the irreducible initial form $c_i \mathbf{x^{a_i}} + c_j \mathbf{x^{a_j}}$. Then by Lemma \ref{RedBinom}, any pair of vertices have $\gcd(\bf{a_i}, \bf{a_j})=1$. We may now apply \cref{VertPoints} to conclude there are no interior lattice points, and consequently all monomial terms in $f$ must correspond to vertices.
\end{proof}

\maintheorem*
\begin{proof}[Proof]
$\Rightarrow$ This is similar to the proof of \cref{interiorlatticepoints}, which by application of the lemmas gives all monomial terms in $f$ must correspond to vertices of $N(f)$, where all vertices are disjointly supported and pairwise of the form that $\gcd(\bf{a_i}, \bf{a_j})=1$.
\\
$\Leftarrow$ By \cref{PolyDis}, the Newton polytope $N(f)$ is a simplex. By \cref{VertPoints} each edge of $N(F)$ is an empty simplex and thus by \cref{RedBinom} correspond to an irreducible binomial. Consider an arbitrary initial form $\ini{\boldsymbol \omega}(f)$ of $f$, and suppose it can be written, $\ini{\boldsymbol \omega}(f) = g_{1}g_{2}$. We show without loss of generality that $g_1$ is a constant. 
We have noted above that any binomial initial form of $f$ is irreducible, so if $\ini{\boldsymbol \omega}(f)$ is a binomial we are finished. The argument for the general case is similar. If $\ini{\boldsymbol \omega}(f) = g_{1}g_{2}$, then we get a Minkowski sum decomposition of $N(\ini{\boldsymbol \omega}(f))=N(g_{1})+N(g_{2})$. We know that $N(\ini{\boldsymbol \omega}(f))$ is a face of $N(f)$ and therefore a simplex. Then, any Minkowski decomposition of a simplex is a sum of a point and a simplex. Then, by the disjoint supports condition, this point must be the origin, and so one of our terms must be constant. 

\end{proof}

\section{The Tropical Variety}\label{Angela}

Let  $f = \sum_{i = 1}^K c_i \x^{\ba_i}$ be a disjointly supported polynomial with no constant term. In this section we explicitly construct the faces of the Gr\"obner fan of the principal ideal $\langle f \rangle$ whose support is $\Trop(f)$.

  Let ${\bf 1}$ be the vector of $1$'s, then $ \ell_i := \langle {\bf 1}, \ba_i \rangle = \sum_{j =1}^n a_i^j > 0$ for all $c_i \neq 0$. 
  By letting $\ell$ be $LCM\{ \ell_i \mid c_i \neq 0\}$ and $\mathbf{v}_f$ be the vector with entry $\frac{\ell}{\ell_i}$ at any index in $supp(\ba_i)$ and 0 otherwise, we see that $\langle \mathbf{v}_f, \ba_i\rangle = \langle\mathbf{v}_f, \ba_j\rangle > 0$ $\forall i, j$. 
  It follows that $f$ is homogeneous with respect to $\mathbf{v}_f$, and that the support of the Gr\"obner fan of $f$ is all of $\R^n$ (see \cite[Proposition 1.12]{Sturmfels96}). 
  We let $L_f$ denote the homogeneity space of $f$, in particular $L_f = \{\bu \mid \langle \bu, \ba_i\rangle = \langle \bu, \ba_j \rangle, \ \ 1 \leq i < j \leq K\}$.  Moreover, for each $\ba_i$ with $c_i \neq 0$ we let $\mathbf{w}_i$ be the vector with entry $0$ for $j \notin supp(\ba_i)$ and $-1$ for $j \in supp(\ba_i)$.

\begin{proposition}
Let $S \subseteq [K]$, $f_S = \sum_{j \in S} c_j \x^{\ba_j}$, and let $C_S := \{\boldsymbol\omega \mid \ini{\boldsymbol\omega}(f) = f_S\}$, then:
\[
C_S = L_f + \sum_{i \notin S} \R_{> 0}\mathbf{w}_i.
\]
\end{proposition}

\begin{proof}
Let $\boldsymbol\omega \in L_f+ \sum_{i \notin S} \R_{> 0}\mathbf{w}_i$ and consider $\ini{\boldsymbol\omega}(f)$. 
Without loss of generality we may assume that $\boldsymbol\omega = \sum_{i \notin S} \; n_i\mathbf{w}_i$ with $n_i > 0$.
The polynomial $f$ has disjoint supports, so $\langle \boldsymbol\omega, \ba_i \rangle$ is $0$ if $i \in S$ and $< 0$ if $i \notin S$. 
It follows that $\ini{\boldsymbol\omega}(f) = f_S$. 
This proves that $L_f + \sum_{i \notin S} \R_{> 0}\mathbf{w}_i \subseteq C_S$.

If $\boldsymbol\omega \in C_S$, then $\boldsymbol\omega$ weights each term of $f_S$ equally. We let $k = \langle \boldsymbol\omega, \ba_j \rangle$, where $j$ is any element of $S$ and $k_i = \langle \boldsymbol\omega, \ba_i\rangle$ for $i \notin S$.  Observe that $k_i < k$, and that 
$\boldsymbol\omega - \sum_{i =1}^K(k-k_i)\frac{1}{\ell_i}\mathbf{w}_i$ weights all monomials of $f$ equally. It follows that $\boldsymbol\omega - \sum (k - k_i)\frac{1}{\ell_i}\mathbf{w}_i \in L_f$. As $k - k_i > 0$, we conclude that $\boldsymbol\omega \in L_f + \sum_{i \notin S} \R_{> 0}\mathbf{w}_i$, and that $C_S \subseteq L_f + \sum_{i \notin S} \R_{> 0}\mathbf{w}_i$.
\end{proof}

Each $f_S$ is a polynomial which corresponds to a face of the Newton polytope of $f$ and likewise, to a cone of the Gr\"obner fan.  Observe that by definition the tropical variety $\Trop(f)$ is the union of the cones $C_S$ where $|S| \geq 2$.

To complete the description of each cone $C_S$ we compute a basis for $L_f$.  First, we observe that $\mathbf{v}_f \in L_f$, and for any $\boldsymbol\lambda \in L_f$ there is some $q$ such that $\langle \boldsymbol\lambda - q\mathbf{v}_f, \ba_i \rangle = 0$ for all $i \in [K]$. The space $N_f = \{\boldsymbol\lambda' \mid \langle \boldsymbol\lambda', \ba_i \rangle = 0 \}$ is certainly contained in $L_f$, so it follows that $\mathbf{v}_f$ and a basis of $N_f$ suffice to give a basis of $L_f$. For a basis of $N_f$ we take the integral vectors $\mathbf{v}_{i,j} = \ba_{i}^{1}e_{i}^{j} - \ba_{i}^{j}e_{i}^{j}$; for $2 \leq j \leq k_i$, where $e_{i}^{j}$ is the $j$-th elementary basis vector from the support of $\ba_i$. 

Theorem 1 of \cite{Kaveh-Manon-NOK} gives a recipe for producing a full rank valuation with associated Newton-Okounkov given a prime cone from a tropical variety.  It is required to choose a a linearly independent set of vectors from the cone which span a full dimensional subcone.  For the cone $C_S$ with $|S| = 2$ we select the set $W_S$ composed of the basis $\{\mathbf{v}_f, \ldots, \mathbf{v}_{i,j}, \ldots\} \subset L_f$, and the extremal vectors $\mathbf{w}_i$ for $i \in S^c$.  If the sets $S$ and $S'$ differ by a single index, then $W_S$ and $W_{S'}$ differ by a single vector.  We let $M_S$ be the matrix with rows equal to the elements of $W_S$.

\[\begin{bmatrix}
\mathbf{v}_f\\
\vdots\\
\mathbf{v}_{2,i}\\
\mathbf{v}_{3, i}\\
\vdots\\
\mathbf{v}_{k_i,i}\\
\vdots\\
\mathbf{w}_{i}\\
\vdots
\end{bmatrix} = \begin{bmatrix}
\cdots & \frac{\ell}{\ell_i} & \frac{\ell}{\ell_i} & \frac{\ell}{\ell_i} & \cdots & \frac{\ell}{\ell_i} & \cdots \\
\cdots & \vdots & \vdots & \vdots & \vdots & \vdots & \cdots \\
\cdots & \ba_{2}^i & -\ba_{1}^i& 0 & \cdots & 0 & \cdots\\
\cdots & \ba_{3}^i & 0 & -\ba_{1}^i & \cdots & 0 & \cdots \\
\cdots & \vdots & \vdots & \vdots & \ddots & \vdots & \cdots \\
\cdots & \ba_{k_i}^i & 0 & 0 & \cdots & -\ba_{1}^i & \cdots \\
\cdots & \vdots & \vdots & \vdots & \vdots & \vdots & \cdots \\
\cdots & -1 & -1 & -1 & \cdots & -1 & \cdots \\
\cdots & \vdots & \vdots & \vdots & \vdots & \vdots & \cdots \\
\end{bmatrix}\]

By \cite[Proposition 4.2]{Kaveh-Manon-NOK}, the matrix $M_S$ defines a full rank valuation $\mathfrak{v}_S : \K[\x]/\langle f \rangle \to \R^{n-1}$. The image $S(\K[\x]/\langle f \rangle, \mathfrak{v}_S) \subseteq \R^{n-1}$ of $\mathfrak{v}_S$ is the semigroup generated by the columns of $M_S$ under addition, where $\mathfrak{v}_S(x_{ij})$ is the $ij$-th column of $M_S$.  If $i \in S^c$, $\mathfrak{v}_S$ sends $x_{ij}$ to the $j$-th column of the block displayed above.  If $i \in S$, $\mathfrak{v}_S$ sends $x_{ij}$ to the $j$-th column of the block displayed above, except the $-1$ entries are $0$.  

The convex hull $P(\K[\x]/\langle f \rangle, \mathfrak{v}_S) \subseteq \R^{n-1}$ of $S(\K[\x]/\langle f \rangle, \mathfrak{v}_S) \subseteq \R^{n-1}$ is called the \emph{Newton-Okounkov} cone of $\mathfrak{v}_S$.  For each choice of $S$ with $|S|=2$, there is a flat family $\pi_S: E_S \to \mathbb{A}^1(\K)$ such that the coordinate ring of the fiber $\pi_S^{-1}(c)$ for $c \neq 0$ is $\K[\x]/\langle f \rangle$ and the coordinate ring of the fiber $\pi_S^{-1}(0)$ is the affine semigroup algebra $\K[S(\K[\x]/\langle f \rangle, \mathfrak{v}_S)]$.  In particular, $\K[S(\K[\x]/\langle f \rangle, \mathfrak{v}_S)] \cong \K[\x]/\langle \ini{\boldsymbol \omega}(f) \rangle$ for any $\boldsymbol \omega \in C_S$.

If there is a vector ${\bf d} \in \Z_{> 0}^n$ such that $\langle {\bf d}, \ba_i \rangle$ is a fixed integer for all monomial exponents $\ba_i$ appearing in $f$, we say that $f$ is homogeneous with respect to ${\bf d}$. For example, if $f$ is homogeneous in the classical sense we may take ${\bf d}$ to be the all $1$'s vector.  Assuming a fixed ${\bf d}$, the algebra $\K[\x]/\langle f \rangle$ is positively graded, that is it can be expressed as direct sum of finite dimensional vector spaces $A_N$:

\[\K[\x]/\langle f \rangle \cong \bigoplus_{N \geq 0} A_N.\]

\noindent
In this setting, the projective variety $X = \textup{Proj}(\K[\x]/\langle f \rangle)$ carries a flat degeneration to the projective toric variety $X_S = \textup{Proj}(\K[S(\K[\x]/\langle f \rangle, \mathfrak{v}_S)])$. It is possible that $X_S$ is non-normal, however the normalization is the projective toric variety associated to the \emph{Newton-Okounkov body} $\Delta(\K[\x]/\langle f \rangle, \mathfrak{v}_S)$.  Following \cite[Corollary 4.7]{Kaveh-Manon-NOK}, the Newton-Okounkov body $\Delta(\K[\x]/\langle f \rangle, \mathfrak{v}_S)$ is obtained from $M_S$ by dividing the $ij$-th column by the degree of $x_{ij}$ assigned by ${\bf d}$, and taking the convex hull of the resulting column vectors. 

\begin{example}
\emph{
We compute the matrices $M_S$ for $f = x + y^2 + zw \in \K[x, y, z, w]$. Let $x, y^2$, and $zw$ be the $i =1, 2$ and $3$ monomials, respectively.  The space $L_f \subset \mathbb{Q}^4$ can be generated by the vectors ${\bf v}_f = (2, 1, 1, 1)$ and ${\bf v}_{2,3} = (0, 0, 1, -1)$. From this we deduce that $A = \K[x, y, z, w]/\langle f \rangle$ is graded by the semigroup in $\Z^2$ generated by $(1, 0), (1, 1),$ and $(1, -1)$.  The third row of $M_S$ is ${\bf w}_i$, where $\{i\} = S^c$:  
}
\[M_{12} = \begin{bmatrix}
2 & 1 & 1 & 1\\
0 & 0 & 1 & -1\\
0 & 0 & -1 & -1
\end{bmatrix}\hfill M_{13} = \begin{bmatrix}
2 & 1 & 1 & 1\\
0 & 0 & 1 & -1\\
0 & -1 & 0 & 0
\end{bmatrix}\hfill M_{23} = \begin{bmatrix}
2 & 1 & 1 & 1\\
0 & 0 & 1 & -1\\
-1 & 0 & 0 & 0
\end{bmatrix}\]
\emph{
\noindent
For each $S$, the columns of $M_S$ generate a semigroup in $\Z^3$ whose semigroup algebra is the coordinate ring of a toric degeneration of $\K[x, y, z, w]/\langle f \rangle$.
}
\end{example}

\bibliography{geo}

\begin{thebibliography}{10}

\bibitem{Arz}
I.V. Arzhantsev.
\newblock On the factoriality of cox rings.
\newblock Math Notes 85, 623–629 (2009).

\bibitem{Cox-book}
Ivan Arzhantsev, Ulrich Derenthal, Juergen Hausen, and Antonio Laface.
\newblock {\em Cox Rings}.
\newblock Cambridge Studies in Advanced Mathematics 144. Cambridge University
  Press, 2014.

\bibitem{Cummings-Manon}
Joseph Cummings and Christopher Manon.
\newblock Well-poised embeddings of affine arrangement varieties.
\newblock arXiv:2009.09105 [math.AG].

\bibitem{Escobar-Harada}
Megumi Harada and Laura Escobar.
\newblock Wall-crossing for newton-okounkov bodies and the tropical
  grassmannian.
\newblock arXiv:1912.04809 [math.AG].

\bibitem{HHW}
Jurgen Hausen, Christoff Hische, and Milena Wrobel.
\newblock On torus actions of higher complexity.
\newblock Forum of Mathematics, Sigma (2019), Vol. 7, e38, 81 pages.

\bibitem{Ilten-Manon}
Nathan Ilten and Christopher Manon.
\newblock Rational complexity-one {T}-varieties are well-poised.
\newblock Int. Math. Res. Not, rnx254, 2017.

\bibitem{KK}
Kiumars Kaveh and A.~G. Khovanskii.
\newblock Newton-{O}kounkov bodies, semigroups of integral points, graded
  algebras and intersection theory.
\newblock {\em Ann. of Math. (2)}, 176(2):925--978, 2012.

\bibitem{Kaveh-Manon-NOK}
Kiumars Kaveh and Christopher Manon.
\newblock Khovanskii bases, higher rank valuations, and tropical geometry.
\newblock {\em SIAM J. Appl. Algebra Geom.}, 3(2):292--336, 2019.

\bibitem{LM}
Robert Lazarsfeld and Mircea Mustata.
\newblock Convex bodies associated to linear series.
\newblock {\em Ann. Sci. \'Ec. Norm. Sup\'er. (4)}, 42(5):783--835, 2009.

\bibitem{Maclagan-Sturmfels}
Diane Maclagan and Bernd Sturmfels.
\newblock {\em Introduction to tropical geometry}, volume 161 of {\em Graduate
  Studies in Mathematics}.
\newblock American Mathematical Society, Providence, RI, 2015.

\bibitem{Ok}
Andrei Okounkov.
\newblock Why would multiplicities be log-concave?
\newblock In {\em The orbit method in geometry and physics ({M}arseille,
  2000)}, volume 213 of {\em Progr. Math.}, pages 329--347. Birkh\"auser
  Boston, Boston, MA, 2003.

\bibitem{Sturmfels96}
Bernd Sturmfels.
\newblock {\em Gr\"obner bases and convex polytopes}, volume~8 of {\em
  University Lecture Series}.
\newblock American Mathematical Society, Providence, 1996.

\end{thebibliography}
\bibliographystyle{plain}

\end{document}